\documentclass[11pt,a4paper,twoside]{article}
\usepackage{amsthm,amsfonts,amsmath}

\newtheorem{theorem}{Theorem}[section]
\newtheorem{corollary}{Corollary}[section]
\newtheorem{definition}{Definition}[section]
\newtheorem{example}{Example}[section]

\newtheorem{proposition}{Proposition}[section]
\newtheorem{remark}{Remark}[section]

\author{
       Vladimir Rovenski\footnote{Department of Mathematics, University of Haifa, Mount Carmel, 3498838 Haifa, Israel
       \newline e-mail: {\tt vrovenski@univ.haifa.ac.il}}}

\title{On the geometry of a weakened $f$-structure}

\begin{document}

\date{}

\maketitle

\begin{abstract}
An $f$-structure, introduced by K.\,Yano in 1963 and subsequently studied by a number of geometers,
is a higher dimensional analog of almost complex and almost contact structures,
defined by a (1,1)-tensor field $f$ on a $(2n+p)$-dimensional manifold, which satisfies $f^3 + f = 0$ and has constant rank $2n$.
We recently introduced the weakened (globally framed) $f$-structure
(i.e., the complex structure on $f(TM)$ is replaced by a nonsingular skew-symmetric tensor)
and its subclasses of weak $K$-, ${\cal S}$-, and ${\cal C}$- structures
on Riemannian manifolds with totally geodesic foliations, which allow us to take a fresh look at the classical theory.
We demonstrate this by generalizing several known results on globally framed $f$-manifolds.
First, we express the covariant derivative of $f$ using a new tensor on a metric weak $f$-structure,
then we prove that on a weak $K$-manifold the characteristic vector fields are Killing and $\ker f$ defines a totally geodesic foliation,
an ${\cal S}$-structure is rigid, i.e., our weak ${\cal S}$-structure is an ${\cal S}$-structure,
and a metric weak $f$-structure with parallel tensor $f$ reduces to a weak~${\cal C}$-structure.
For $p=1$ we obtain the corresponding corollaries for weak almost contact, weak cosymplectic, and weak Sasakian structures.


\end{abstract}




\section*{Introduction}

In 1963, K.\,Yano \cite{yan} introduced an $f$-structure on a smooth manifold $M^{2n+p}$,
as a higher dimensional analog of almost complex ($p=0$) and almost contact ($p=1$) structures,
defined by a
(1,1)-tensor field $f$, which satisfies $f^3 + f = 0$ and has constant rank $2n$.
In this case, the tangent bundle $TM$ splits into two complementary subbundles: a $2n$-dimensional ${\rm im}\,f$ and a $p$-dimensional $\ker f$,
and the projection maps are $-f^2$ and $f^2 + 1$.
Moreover, the restriction of $f$ to ${\rm im}\,f$ determines a complex structure,
and the structure group of $TM$ reduces to $U(n)\times O(p)$.
An~interesting case studied by a number of geometers \cite{b1970,BP-2008A,DIP-2001,fip,gy,tpw} occurs when $\ker f$ is parallelizable
or is defined by a homomorphism of a $p$-dimensional Lie algebra $\mathfrak{g}$ to the Lie algebra of all vector fields on $M$,
i.e., $M$ admits a $\mathfrak{g}$-foliation, see~\cite{AM-1995}.
In the presence of a compatible metric, such $\mathfrak{g}$-foliation is a totally geodesic foliation spanned by Killing vector fields, e.g., \cite{ca-toh}.
For a 1-dimensional Lie algebra, a ${\mathfrak{g}}$-foliation is generated by a nonvanishing vector~field,
and we get (almost) contact metric structures as well as $K$-contact and Sasakian ones, see~\cite{blair2010riemannian}.

We recently introduced \cite{RWo-2} the weakened metric structures that generalize an $f$-structure and its satellites,
i.e., the complex structure on ${\rm im}\,f$ is replaced by a nonsingular skew-symmetric tensor, see \eqref{2.1} below,
and allow us to take a fresh look at the classical theory.
In~this paper, we continue the study of a weak $f$-structure associated with a Riemannian manifold endowed with a totally geodesic foliation,
and generalize several known results.
A~natu\-ral question arises: \textit{How rich are metric weak $f$-structures compared to the classical ones}?
We~study this question for weak ${K}$-, ${\cal S}$- and ${\cal C}$- structures;
the~results can be compared with corresponding results ($p=1$) in \cite{RovP-arxiv} for weak almost contact, weak cosymplectic, and weak Sasakian structures.
The~question for a metric weak $f$-structure in general and further investigation of the curvature of metric $f$-manifolds
and their connection with such popular structures as warped products and Ricci solitons are postponed to the future.
The~theory presented here can be used to deepen our knowledge of Riemannian geometry of
manifolds equipped with distributions, foliations and submersions.

This article consists of an introduction and five sections.
In Section~\ref{sec:1}, we discuss the properties of ``weak" metric structures generalizing certain classes of $f$-manifolds.
In Section~\ref{sec:2} we express the covariant derivative of $f$ of a metric weak $f$-structure using a new tensor
and show that on a weak $K$-manifold the characteristic vector fields are Killing and $\ker f$ defines a totally geodesic foliation.
In Section~\ref{sec:3a}, we apply to weak almost ${\cal S}$-manifolds the tensor~$h$.
In Section~\ref{sec:3} we complete the result in \cite{RWo-2}
and prove the rigidity theorem that any weak ${\cal S}$-structure is an ${\cal S}$-structure.
In Section~\ref{sec:4}, we show that a metric weak $f$-structure with parallel tensor $f$ reduces to a weak ${\cal C}$-structure,
we also give an example of such a structure on the product of~manifolds.

\section{Preliminaries}
\label{sec:1}

Here, we describe ``weak" metric structures generalizing certain classes of $f$-manifolds and discuss their properties.
A \textit{weak $f$-structure} on a smooth manifold $M^{2n+p}$ is defined by a $(1,1)$-tensor field $f$ of rank $2n$
and a~nonsingular $(1,1)$-tensor field $Q$ satisfying, see \cite{RWo-2},
\begin{eqnarray}\label{E-fQ-1}
 && f^3 + fQ = 0,\\
 \label{2.1Q-nu}
 && Q\,\xi=\xi\quad (\xi\in\ker f).
\end{eqnarray}
Given such a structure, the tangent bundle $TM$ splits into two complementary subbundles $f(TM)$ and the kernel $\ker f$
(also called characteristic distribution).
If $\ker f$ is parallelizable, then we fix $p$ global vector fields $\xi_i\ (1\le i\le p)$, which span $\ker f$,
and dual one-forms $\eta^i$. Then we have
\begin{equation}\label{2.1}
 f^2 = -Q +\sum\nolimits_{\,i}\eta^i\otimes\xi_i, \quad \eta^i(\xi_j)=\delta^i_j ,
\end{equation}
called a \textit{framed weak $f$-structure} (or a \textit{weak f.pk-structure}).
For $p = 1$, this gives a weak almost contact structure, see~\cite{RWo-2}.

\begin{remark}\rm
The concept of an almost paracontact structure is analogous to the concept of an almost contact structure and is closely related to an almost product structure.
Similarly to \eqref{E-fQ-1}, we can define a \textit{weak para}-$f$-\textit{structure} by $f^3 - fQ = 0$, and for parallelizable $\ker f$ use
$f^2 = Q -\sum\nolimits_{\,i}\eta^i\otimes\xi_i$ instead of \eqref{2.1},
 see \cite{RWo-2} and \cite[Section~5.3.8]{Rov-Wa-2021}.
This allows us to generalize (as in Sections~\ref{sec:2}--\ref{sec:4} below) some results on para-$f$-structures introduced in~\cite{BN-1985,FP-2017}.
\end{remark}

Using \eqref{2.1}$_2$ we get $f(TM)=\bigcap_{\,i}\ker\eta^i$, and using \eqref{2.1}$_1$ we conclude that $f(TM)$ is invariant for~$f$, i.e.,
\begin{equation}\label{2.1-D}
 {f} X\in f(TM)\quad (X\in f(TM)).
\end{equation}
By \eqref{2.1} and \eqref{2.1-D}, $f(TM)$ is also invariant for $Q$.
 We say a framed weak $f$-structure is \textit{normal} if the following tensor
 (known for $Q={\rm id}_{\,TM}$, e.g., \cite{BP-2008A,DIP-2001}, or \cite{b1970} without factor 2) is identically zero:
\begin{align}\label{2.6X}
 N^{(1)} = [{f},{f}] + 2\sum\nolimits_{\,i} d\eta^i\otimes\,\xi_i .
\end{align}
The Nijenhuis torsion
of ${f}$ and the exterior derivative
of $\eta^i$ are given~by
\begin{align}\label{2.5}
 [{f},{f}](X,Y) & = {f}^2 [X,Y] + [{f} X, {f} Y] - {f}[{f} X,Y] - {f}[X,{f} Y],\ X,Y\in\mathfrak{X}_M , \\
\label{3.3A}
 d\eta^i(X,Y) & = \frac12\,\{X(\eta^i(Y)) - Y(\eta^i(X)) - \eta^i([X,Y])\},\quad X,Y\in\mathfrak{X}_M .
\end{align}

\begin{remark}\rm
Consider the product manifold $\bar M = M^{2n+p}\times\mathbb{R}^p$,
where $\mathbb{R}^p$ is a Euclidean space with an orthonormal basis $\partial_1,\ldots,\partial_p$,
and define tensor fields $\bar f$ and $\bar Q$ on $\bar M$ putting
\begin{align*}
 \bar f(X,\, \sum\nolimits_{\,i} a^i\partial_i) &= (fX-\sum\nolimits_{\,i} a^i\xi_i,\, \sum\nolimits_{\,j} \eta^j(X)\partial_j),\\
 \bar Q(X,\, \sum\nolimits_{\,i} a^i\partial_i) &= (QX,\, \sum\nolimits_{\,i} a^i\partial_i) .
\end{align*}
Hence, $\bar f(X,0)=(fX,0)$, $\bar Q(X,0)=(QX,0)$ for $X\in\ker f$,
$\bar f(\xi_i,0)=(0,\partial_i)$, $\bar Q(\xi_i,0)=(\xi_i,0)$
and
$\bar f(0,\partial_i)=(-\xi_i,0)$, $\bar Q(0,\partial_i)=(0,\partial_i)$.
Then it is easy to verify that $\bar f^{\,2}=-\bar Q$.
One can use the integrability condition $[\bar f, \bar f]=0$ of $\bar f$ to express the normality of a framed weak $f$-structure.
\end{remark}

If there exists a semi-Riemannian metric $g$ such that
$g(\xi_i,\xi_i)=\epsilon_i\in\{-1, 1\}$, $g_{\,|\,f(TM)}>0$ and
\begin{align}\label{2.2}
 g({f} X,{f} Y)= g(X,Q\,Y) -\sum\nolimits_{\,i}\epsilon_i\,\eta^i(X)\,\eta^i(Q\,Y),\quad X,Y\in\mathfrak{X}_M,
\end{align}
then $({f},Q,\xi_i,\eta^i,g)$ is called a {\it metric weak $f$-structure},
$M({f},Q,\xi_i,\eta^i,g)$ is called a \textit{metric weak $f$-manifold}, and $g$ is called a \textit{compatible metric}.
 Putting $Y=\xi_i$ in \eqref{2.2} and using \eqref{2.1Q-nu}, we get
\begin{align*}
  g(X,\xi_i) = \epsilon_i\,\eta^i(X),
\end{align*}
thus, $f(TM)\,\bot\,\ker f$ and $\{\xi_i\}$ is an orthonormal frame of $\ker f$.
For simplicity, we will assume $\epsilon_i=1$.

\begin{remark}\rm
According to \cite{RWo-2}, a framed weak $f$-structure admits a compatible metric if ${f}$
admits a skew-symmetric representation, i.e., for any $x\in M$ there exist a neighborhood $U_x\subset M$ and a~frame $\{e_k\}$ on $U_x$,
for which ${f}$ has a skew-symmetric matrix.
\end{remark}

The following statement is well-known for the case of $Q={\rm id}_{\,TM}$.

\begin{proposition}
{\rm (a)}
For a framed weak $f$-structure the following equalities hold:
\[
 {f}\,\xi_i=0,\quad \eta^i\circ{f}=0\quad (1\le i\le p),\quad [Q,\,{f}]=0.
\]
{\rm (b)}
For $M({f},Q,\xi_i,\eta^i,g)$, the tensor ${f}$ is skew-symmetric and the tensor $Q$ is self-adjoint, i.e.,
\begin{equation}\label{E-Q2-g}
 g({f} X, Y) = -g(X, {f} Y),\quad
 g(QX,Y)=g(X,QY).
\end{equation}
\end{proposition}

\begin{proof}
(a) By \eqref{2.1Q-nu} and \eqref{2.1}, ${f}^2\xi_i=0$.
Applying \eqref{E-fQ-1} to $f\xi_i$, we get $f\xi_i=0$.
To show $\eta^i\circ{f}=0$, note that $\eta^i({f}\,\xi_i)=\eta^i(0)=0$, and, using \eqref{2.1-D}, we get $\eta^i({f} X)=0$ for $X\in f(TM)$.
Next, using \eqref{2.1} and ${f}(Q\,\xi_i) = {f}\,\xi_i=0$, we get
\begin{align*}
 {f}^3 X & = {f}({f}^2 X) = -{f}\,QX +\sum\nolimits_{\,i}\eta^i(X)\,{f}\xi_i = -{f}\,QX,\\
 {f}^3 X & = {f}^2({f} X) = -Q\,{f} X +\sum\nolimits_{\,i}\eta^i({f} X)\,\xi_i = -Q\,{f} X
\end{align*}
for any $X\in f(TM)$. This and $[Q,\,{f}]\,\xi_i=0$ provide $[Q,\,{f}]=Q\,{f} - {f} Q = 0$.

(b) By~\eqref{2.2}, the~restriction $Q_{\,|\,f(TM)}$ is self-adjoint. This and \eqref{2.1Q-nu} provide \eqref{E-Q2-g}$_2$.
For any $Y\in f(TM)$ there is $\tilde Y\in f(TM)$ such that ${f}Y=\tilde Y$.
Thus, \eqref{E-Q2-g}$_1$ follows from \eqref{2.2} and \eqref{E-Q2-g}$_2$ with $X\in f(TM)$ and $\tilde Y$.
\end{proof}

\begin{remark}\rm
The Levi-Civita connection $\nabla$ of a semi-Riemannian metric $g$ is given by
\begin{align}\label{3.2}
 2\,g(\nabla_{X}Y,Z) &= X\,g(Y,Z) + Y\,g(X,Z) - Z\,g(X,Y) \notag\\
 &+ g([X,Y],Z) +g([Z,X],Y) - g([Y,Z],X),
\end{align}	
and has the properties
$X\,g(Y,Z)=g(\nabla_X\,Y,Z)+g(Y,\nabla_X\,Z)$ (metric compa\-tible)
and
$[X,Y]=\nabla_X\,Y-\nabla_Y\,X$ (without torsion).
Thus, \eqref{2.5} can be written in terms of $\nabla$ as
\begin{align}\label{4.NN}
 [{f},{f}](X,Y) = ({f}\nabla_Y{f} - \nabla_{{f} Y}{f}) X - ({f}\nabla_X{f} - \nabla_{{f} X}{f}) Y;
\end{align}
in particular, since ${f}\,\xi_i=0$,
\begin{align}\label{4.NNxi}
 [{f},{f}](X,\xi_i)= {f}(\nabla_{\xi_i}{f})X +\nabla_{{f} X}\,\xi_i -{f}\,\nabla_{X}\,\xi_i, \quad X\in \mathfrak{X}_M .
\end{align}
\end{remark}

The {fundamental $2$-form} $\Phi$ on $M({f},Q,\xi_i,\eta^i,g)$ is defined by
\begin{align*}
 \Phi(X,Y)=g(X,{f} Y),\quad X,Y\in\mathfrak{X}_M.
\end{align*}
Since $\eta^1\wedge\ldots\wedge\eta^p\wedge\Phi^n\ne0$,
a metric weak ${f}$-manifold is orientable.

On a normal framed $f$-manifold with a compatible metric and $d\Phi=0$ (called $K$-manifold) the vector fields $\xi_i$ are Killing, e.g., \cite[Theorem~1.1]{b1970}, and we have $\nabla_{\xi_i}\,\xi_j=0$.
We get $\nabla_{X}\,\xi_i=-\frac12\,fX$ on ${\cal S}$-manifolds, and $\nabla_{X}\,\xi_i=0$ on ${\cal C}$-manifolds, e.g., \cite[Lemma~1.2]{b1970}.
In the framework of Riemannian geometry, ${\cal C}$- and ${\cal S}$- structures
generalize cosymplectic and Sasaki structures, respectively.

\begin{definition}\rm
A metric weak $f$-structure $({f},Q,\xi_i,\eta^i,g)$ is called a \textit{weak $K$-structure} if it is normal and the form $\Phi$ is closed,
i.e., $d \Phi=0$. We define two subclasses of weak $K$-manifolds as follows:
\textit{weak ${\cal C}$-manifolds} if $d\eta^i = 0$ for any $i$, and \textit{weak ${\cal S}$-manifolds}~if
\begin{align}\label{2.3}
 d\eta^i = \Phi,\quad 1\le i\le p .
\end{align}
Omitting the normality condition, we get the following:
a metric weak $f$-structure
is called
(i)~a \textit{weak almost ${\cal S}$-structure} if \eqref{2.3} is valid;
(ii)~a \textit{weak almost ${\cal C}$-structure}
if $\Phi$ and $\eta^i$ are closed forms.
\end{definition}

For $p=1$, weak ${\cal C}$- and weak ${\cal S}$- manifolds reduce to weak (almost) cosymplectic manifolds and weak (almost) Sasakian manifolds, respectively, see \cite{RWo-2}.

Recall the following formulas with
the Lie derivative $\pounds_{\xi_i}$ in the $Z$-direction and $X,Y\in\mathfrak{X}_M$:
\begin{align}\label{3.3B}
 (\pounds_{Z}{f})X &  = [Z, {f} X] - {f} [Z, X],\\
\label{3.3C}
 (\pounds_{Z}\,\eta^j)X & = Z(\eta^j(X)) - \eta^j([Z, X]) , \\
\label{3.7}
\nonumber
 (\pounds_{Z}\,g)(X,Y) &= Z(g(X,Y)) - g([Z,X], Y) - g(X, [Z,Y]) \\
 & = g(\nabla_{X}\,Z, Y) + g(\nabla_{Y}\,Z, X).
\end{align}
The following tensors are well known in the theory of $f$-manifolds, see \cite{blair2010riemannian}:
\begin{align}
\label{2.7X}
 N^{(2)}_i(X,Y) &= (\pounds_{{f} X}\,\eta^i)Y - (\pounds_{{f} Y}\,\eta^i)X \overset{\eqref{3.3A}}= 2\,d\eta^i({f} X,Y) - 2\,d\eta^i({f} Y,X),  \\
\label{2.8X}
 N^{(3)}_i(X) &= (\pounds_{\xi_i}{f})X \overset{\eqref{3.3B}}= [\xi_i, {f} X] - {f} [\xi_i, X],\\
\label{2.9X}
 N^{(4)}_{ij}(X) &= (\pounds_{\xi_i}\,\eta^j)X \overset{\eqref{3.3C}}= \xi_i(\eta^j(X)) - \eta^j([\xi_i, X])
 = 2\,d\eta^j(\xi_i, X).
\end{align}
For $p=1$, the tensors \eqref{2.7X}--\eqref{2.9X} reduce to the following tensors on (weak) almost contact manifolds:
\begin{align*}
 N^{(2)}(X,Y) = (\pounds_{\varphi X}\,\eta)Y - (\pounds_{\varphi Y}\,\eta)X, \quad
 N^{(3)} = \pounds_{\xi}\,\varphi,\quad
 N^{(4)} = \pounds_{\xi}\,\eta .
\end{align*}

\section{The geometry of a metric weak $f$-structure}
\label{sec:2}

Here, we study the geometry of the characteristic distribution $\ker f$,
supplement the traditional sequence of tensors \eqref{2.6X} and \eqref{2.7X}--\eqref{2.9X}
with a new tensor $N^{(5)}$ and calculate the covariant derivative of $f$ on a metric weak $f$-structure.

A distribution ${\cal D}\subset TM$ is \textit{totally geodesic} if and only if
its second fundamental form vanishes, i.e., $\nabla_X Y+\nabla_Y X\in{\cal D}$ for any vector fields $X,Y\in{\cal D}$ --
this is the case when {any geodesic of $M$ that is tangent to ${\cal D}$ at one point is tangent to ${\cal D}$ at all its points},
e.g., \cite[Section~1.3.1]{Rov-Wa-2021}. Any integrable and totally geodesic distribution determines a totally geodesic foliation.
A foliation, whose orthogonal distribution is totally geodesic, is said to be a Riemannian foliation, e.g., \cite{Rov-Wa-2021}.
For example, a foliation is Riemannian if it is generated by Killing vector fields.

Define a ``small" (1,1)-tensor $\tilde Q$ (vanishing on an $f$-structure) by
\begin{equation*}
 \tilde{Q} = Q - {\rm id}_{\,TM}.
\end{equation*}
We obtain $[\tilde{Q},{f}]=0$, see \eqref{2.1Q-nu},
and also ${f}^3+{f} = -\tilde{Q}{f}$.

Note that $X = X^\top + X^\bot$, where
$X^\top = X - \sum_{\,i}\eta^i(X)\,\xi_i$ is the projection of the vector $X\in TM$ onto $f(TM)$.

The following statement is based on \cite[Theorem~6.1]{blair2010riemannian}, i.e., $Q={\rm id}_{\, TM}$.

\begin{proposition}\label{thm6.1}
Let a metric weak $f$-structure be normal. Then $N^{(3)}_i$ and $N^{(4)}_{ij}$ vanish~and
\begin{align}\label{3.1KK}
 N^{(2)}_i(X,Y) =\eta^i([\tilde QX^\top,\,{f} Y]);
\end{align}
moreover, the characteristic distribution $\ker f$ is totally geodesic.
\end{proposition}

\begin{proof}
Assume $N^{(1)}(X,Y)=0$ for any $X,Y\in TM$. Taking $\xi_i$ instead of $Y$ and using the formula of Nijenhuis tensor \eqref{2.5}, we~get
\begin{align}\label{3.11}
 0 &= [{f},{f}](X,\xi_i) + 2\sum\nolimits_{\,j} d\eta^j(X,\xi_i)\,\xi_j \notag\\
 &= {f}^2[X,\xi_i] - {f}[{f} X,\xi_i] + 2\sum\nolimits_{\,s} d\eta^s(X,\xi_i)\,\xi_s.
\end{align}
Taking the scalar product of \eqref{3.11} with $\xi_j$ and using
skew-symmetry of ${f}$ and ${f}\,\xi_i=0$, we~get
\begin{align}\label{3.11A}
 d\eta^j(\xi_i,\,\cdot)=0\quad ({\rm or},\ \iota_{\,\xi_i} d\eta^j=0);
\end{align}
hence, $N^{(4)}_{ij}=0$, see \eqref{2.9X}.
Next, combining \eqref{3.11} and \eqref{3.11A}, we get
\begin{align*}
 0 = [{f},{f}](X,\xi_i) = {f}^2[X,\xi_i] - {f}[{f} X,\xi_i] = {f}\,(\pounds_{\xi_i}{f})X .
\end{align*}
Applying ${f}$ and using \eqref{2.1} and $\eta^i\circ{f}=0$, we achieve
\begin{align}
\label{3.14}
 0 &= {f}^2 (\pounds_{\xi_i}{f})X
 = -Q(\pounds_{\xi_i}{f})X + \sum\nolimits_{\,j}\eta^j((\pounds_{\xi_i}{f})X)\,\xi_j \notag\\
 &=  -Q(\pounds_{\xi_i}{f})X + \sum\nolimits_{\,j}\eta^j([\xi_i,{f} X])\,\xi_j.
\end{align}
Further, \eqref{3.11A} and \eqref{3.3A} yield
\begin{align}\label{3.11B}
	0=2\,d\eta^j({f} X, \xi_i)
	=({f} X)(\eta^j(\xi_i)) - \xi_i(\eta^j({f} X)) - \eta^j([{f} X, \xi_i])
	=\eta^j([\xi_i, {f} X]).
\end{align}
Since $Q$ is non-singular, from \eqref{3.14}--\eqref{3.11B} we get $\pounds_{\xi_i}{f}=0$, i.e, $N^{(3)}_i=0$, see~\eqref{2.8X}.
	
 Replacing $X$ by ${f} X$ in our assumption $N^{(1)}=0$ and using \eqref{2.5} and \eqref{3.3A}, we acquire
\begin{align}\label{2.6}
 0 &= g([{f},{f}]({f} X,Y) + 2\sum\nolimits_{\,j} d\eta^j({f} X,Y)\,\xi_j,\ \xi_i) \notag\\
 &= g([{f}^2 X,{f} Y],\xi_i) + ({f} X)(\eta^i(Y)) - \eta^i([{f} X,Y]) ,\quad 1\le i\le p.
\end{align}
Using \eqref{2.1} and equality
$[{f} Y,\, \eta^j(X)\,\xi_i] = ({f} Y)(\eta^j(X))\,\xi_i + \eta^j(X)[{f} Y, \xi_i]$, we rewrite \eqref{2.6} as
\begin{align}\label{2.8}\notag
 0 &= -\eta^i([QX, {f} Y]) +\sum\nolimits_{\,j}\eta^j(X)\,\eta^i([\xi_j, {f} Y]) \\
 & - ({f} Y)(\eta^i(X)) + ({f} X)(\eta^i(Y)) - \eta^i([{f} X,Y]).
\end{align}
Equation \eqref{3.11B} gives $\eta^i([{f} Y, \xi_j])=0$. So, \eqref{2.8} becomes
\begin{align}\label{2.9}
 -\eta^i([QX, {f} Y]) - ({f} Y)(\eta^i(X)) + ({f} X)(\eta^i(Y)) - \eta^i([{f} X,Y]) = 0.
\end{align}
Finally, combining \eqref{2.9} with \eqref{2.7X}, we get
\begin{align*}
 N^{(2)}_i(X,Y) =\eta^i([\tilde QX,\,{f} Y]),\quad 1\le i\le p,
\end{align*}
from which and $X = X^\top +\sum\nolimits_{\,i}\eta^i(X)\,\xi_i$ the expression \eqref{3.1KK} of $N^{(2)}_i$ follows.
Using the identity
\begin{align}\label{3.Ld}
 \pounds_{\xi_i}=\iota_{\,{\xi_i}}\,d + d\,\iota_{\,{\xi_i}},
\end{align}
from \eqref{3.11A} and $\eta^i(\xi_j)=\delta^i_j$ we obtain
\begin{align*}
 \pounds_{\xi_i}\,\eta^j = d (\eta^j(\xi_i)) + \iota_{\,\xi_i}\, d\eta^j = 0.
\end{align*}
On the other hand, by \eqref{3.3C} we have
\[
 (\pounds_{\xi_i}\,\eta^j)X= g(X,\nabla_{\xi_i}\,\xi_j)+g(\nabla_{X}\,\xi_i,\,\xi_j),\quad
 X\in\mathfrak{X}_M.
\]
Symmetrizing this and using $\pounds_{\xi_i}\,\eta^j =0$ and $g(\xi_i,\, \xi_j)=\delta_{ij}$ yield
\begin{align}\label{3.30}
 \nabla_{\xi_i}\,\xi_j+\nabla_{\xi_j}\,\xi_i =0,
\end{align}
thus, the distribution $\ker f$ is totally geodesic.
\end{proof}

Recall the co-boundary formula for exterior derivative $d$ on a $2$-form $\Phi$,
\begin{align}\label{3.3}
 d\Phi(X,Y,Z) &= \frac{1}{3}\,\big\{ X\,\Phi(Y,Z) + Y\,\Phi(Z,X) + Z\,\Phi(X,Y) \notag\\
 &-\Phi([X,Y],Z) - \Phi([Z,X],Y) - \Phi([Y,Z],X)\big\}.
\end{align}
By direct calculation we get the following, see also \cite[Proposition~2.5(a)]{BP-2008A} for $Q={\rm id}_{\,TM}$:
\begin{align}\label{3.9A}
 (\pounds_{\xi_i}\,\Phi)(X,Y) = (\pounds_{\xi_i}\,g)(X, {f}Y) + g(X,(\pounds_{\xi_i}{f})Y) .
\end{align}

The following result generalizes \cite[Theorem~1.1]{b1970}.

\begin{theorem}\label{C-K}
On a weak $K$-manifold the vector fields $\xi_1,\ldots,\xi_p$ are Killing;
moreover, $\nabla_{\xi_k}\,\xi_j=0$, i.e., the cha\-racteristic distribution $\ker f$ is integrable and defines a totally geodesic foliation.
\end{theorem}

\begin{proof}
By Proposition~\ref{thm6.1}, the distribution $\ker f$ is totally geodesic and $N^{(3)}_i=\pounds_{\xi_i}{f}=0$.
Using $\iota_{\,{\xi_i}}\Phi=0$ and condition $d\Phi=0$ in the identity \eqref{3.Ld},
we get $\pounds_{\xi_i}\Phi=0$. Thus, from \eqref{3.9A} we obtain $(\pounds_{\xi_i}\,g)(X, {f}Y)=0$.
To show $\pounds_{\xi_i}\,g=0$, we will examine $(\pounds_{\xi_i}\,g)(fX, \xi_j)$ and $(\pounds_{\xi_i}\,g)(\xi_k, \xi_j)$.
Using $\pounds_{\xi_i}\,\eta^j =0$,
we get $(\pounds_{\xi_i}\,g)(fX, \xi_j)=(\pounds_{\xi_i}\,\eta^j)fX -g(fX, [\xi_i,\xi_j])=-g(fX, [\xi_i,\xi_j])=0$.
Next, using \eqref{3.30}, we get
\[
 (\pounds_{\xi_i}\,g)(\xi_k, \xi_j)= -g(\xi_i, \nabla_{\xi_k}\,\xi_j+\nabla_{\xi_j}\,\xi_k) = 0.
\]
Thus, $\xi_i$ is a Killing vector field, i.e., $\pounds_{\xi_i}\,g=0$.
Finally, from $d\Phi(X,\xi_i,\xi_j)=0$ and \eqref{3.3} we obtain $g([\xi_i,\xi_j], fX)=0$, i.e., $\ker f$ is integrable.
\end{proof}

\begin{theorem}\label{thm6.2}
For a weak almost ${\cal S}$-structure,
the tensors $N^{(2)}_i$ and $N^{(4)}_{ij}$ vanish;
moreover, $N^{(3)}_i$ vanishes if and only if $\,\xi_i$ is a Killing vector field.
\end{theorem}

\begin{proof} Applying \eqref{2.3} in \eqref{2.7X} and using skew-symmetry of ${f}$ we get $N^{(2)}_i=0$.
Equation \eqref{2.3} with $Y=\xi_i$ yields $d\eta^j(X,\xi_i)=g(X,{f}\,\xi_i)=0$ for any $X\in\mathfrak{X}_M$; thus, we get \eqref{3.11A},
i.e., $N^{(4)}_{ij}=0$.

Next, invoking \eqref{2.3} in the equality
\begin{align*}
 (\pounds_{\xi_i}\,d\eta^j)(X,Y) = \xi_i(d\eta^j(X,Y)) - d\eta^j([\xi_i,X], Y) - d\eta^j(X,[\xi_i,Y]),
\end{align*}
and using \eqref{3.7}, we obtain for all $i,j$
\begin{align}\label{3.9}
 (\pounds_{\xi_i}\,d\eta^j)(X,Y) = (\pounds_{\xi_i}\,g)(X, {f}Y) + g(X,(\pounds_{\xi_i}{f})Y).
\end{align}
Since $\pounds_V=\iota_{\,V}\circ d+d\circ\iota_{\,V}$, the exterior derivative $d$ commutes with the Lie-derivative, i.e., $d\circ\pounds_V = \pounds_V\circ d$, and as in the proof of Theorem~\ref{C-K}, we get that $d\eta^i$ is invariant under the action of $\xi_i$, i.e., $\pounds_{\xi_i}\,d\eta^j=0$.
Therefore, \eqref{3.9} implies that $\xi_i$ is a Killing vector field if and only if $N^{(3)}_i=0$.
\end{proof}

\begin{theorem}\label{thm6.2C}
For a weak almost ${\cal C}$-structure, we get $N^{(2)}_i=0$, $N^{(4)}_{ij}=0$, $N^{(1)}=[{f},{f}]$, and
\begin{align}\label{6.1d}
 [\xi_i, \xi_j] & = 0,\quad 1\le i,j\le p ,
\end{align}
in particular, the characteristic distribution $\ker f$ is tangent to a totally geode\-sic foliation.
Moreover, $N^{(3)}_i=0\ (1\le i\le p)$ if and only if each $\,\xi_i$ is a Killing vector field.
\end{theorem}

\begin{proof}
By \eqref{2.7X} and \eqref{2.9X} and since $d\eta^i=0$, the tensors $N^{(2)}_i$ and $N^{(4)}_{ij}$ vanish on a weak almost ${\cal C}$-structure.
Moreover, by \eqref{2.6X} and \eqref{3.9}, respectively, the tensor $N^{(1)}$ coincides with $[f,f]$,
and $N^{(3)}_i=\pounds_{\xi_i}{f}\ (1\le i\le p)$ vanish if and only if each $\xi_i$ is a Killing~vector.
From the equalities
\begin{align*}
 3\,d\Phi(X,\xi_i,\xi_j) & = g([\xi_i,\xi_j], fX),\\
 2\,d\eta^k(\xi_j, \xi_i) & = g([\xi_i,\xi_j],\xi_k)
\end{align*}
and conditions $d\Phi=0$ and $d\eta^i=0$ we obtain \eqref{6.1d}.
\end{proof}

We will express $\nabla_{X}{f}$ using a new tensor on a general metric weak $f$-structure.
The following assertion plays a key role in the paper and
generalizes \cite[Proposition~2.4]{BP-2008A}, see also \cite[p.~93]{DIP-2001}.

\begin{proposition}\label{lem6.1}
For a metric weak $f$-structure
we get
\begin{align}\label{3.1}
 & 2\,g((\nabla_{X}{f})Y,Z) = 3\,d\Phi(X,{f} Y,{f} Z) - 3\, d\Phi(X,Y,Z) + g(N^{(1)}(Y,Z),{f} X)\notag\\
 & +\sum\nolimits_{\,i}\big( N^{(2)}_i(Y,Z)\,\eta^i(X) + 2\,d\eta^i({f} Y,X)\,\eta^i(Z) - 2\,d\eta^i({f} Z,X)\,\eta^i(Y)\big) \notag\\
 & + N^{(5)}(X,Y,Z),
\end{align}
where a skew-symmet\-ric with respect to $Y$ and $Z$ tensor $N^{(5)}(X,Y,Z)$ is defined by
\begin{align*}
 N^{(5)}(X,Y,Z) &= ({f} Z)\,(g(X^\top, \tilde QY)) -({f} Y)\,(g(X^\top, \tilde QZ)) \\
 & +\,g([X, {f} Z]^\top, \tilde QY) - g([X,{f} Y]^\top, \tilde QZ) \\
 & +\,g([Y,{f} Z]^\top -[Z, {f} Y]^\top - {f}[Y,Z],\ \tilde Q X).
\end{align*}
\end{proposition}

\begin{proof}
Using \eqref{3.2} and the skew-symmetry of ${f}$, one can compute
\begin{align}\label{3.4}
2\,g((\nabla_{X}{f})Y,Z) =\,& 2\,g(\nabla_{X}({f} Y),Z) + 2\,g( \nabla_{X}Y,{f} Z) \notag\\
=\,& X\,g({f} Y,Z) + ({f} Y)\,g(X,Z) - Z\,g(X,{f} Y) \notag\\
& +\,g([X,{f} Y],Z) +g([Z,X],{f} Y) - g([{f} Y,Z],X) \notag\\
& +\,X\,g(Y,{f} Z) + Y\,g(X,{f} Z) - ({f} Z)\,g(X,Y) \notag\\
& +\,g([X,Y],{f} Z) + g([{f} Z,X],Y) - g([Y,{f} Z],X).
\end{align}
Using \eqref{2.2}, we obtain
\begin{align}\label{XZ}
\notag
 g(X,Z) &= \Phi({f} X, Z) -g(X,\tilde Q Z) +\sum\nolimits_{\,i}\big(\eta^i(X)\,\eta^i(Z) +\eta^i(X)\,\eta^i(\tilde Q Z)\big)\\
 &= \Phi({f} X, Z) + \sum\nolimits_{\,i}\eta^i(X)\,\eta^i(Z)  - g(X^\top, \tilde QZ).
\end{align}
Thus, and in view of the skew-symmetry of ${f}$ and applying \eqref{XZ} six times, \eqref{3.4} can be written as
\begin{align*}
& 2\,g((\nabla_{X}{f})Y,Z) = X\,\Phi(Y, Z) \\
& +({f} Y)\,\big(\Phi({f} X, {Z})+\sum\nolimits_{\,i}\eta^i(X)\,\eta^i(Z) \big) - ({f} Y)\,g(X^\top,\tilde QZ) - Z\,\Phi(X,Y) \\
& -\Phi([X,{f} Y],{f} {Z}) + \sum\nolimits_{\,i}\eta^i([X,{f} Y])\eta^i(Z) - g([X,{f} Y]^\top,\tilde QZ) +\Phi([Z,X],Y) \notag\\
& +\Phi([{f} Y,Z],{f} {X}) - \sum\nolimits_{\,i}\eta^i([{f} Y,Z])\,\eta^i(X) + g([{f} Y, Z]^\top, \tilde QX) + X\,\Phi(Y,Z) \\
& +Y\,\Phi(X,Z) - ({f} Z)\,\big(\Phi({f} X, {Y}) + \sum\nolimits_{\,i}\eta^i(X)\,\eta^i(Y)\big) + ({f} Z) g(X^\top, \tilde QY) \\
& +\Phi([X,Y],Z) + g({f}[{f} Z,X],{f} {Y}) + \sum\nolimits_{\,i}\eta^i([{f} Z,X])\eta^i(Y) - g([{f} Z,X]^\top,\tilde QY)\\
& -g({f}[Y,{f} Z],{f} {X}) - \sum\nolimits_{\,i}\eta^i([Y,{f} Z])\,\eta^i(X) + g([Y,{f} Z]^\top, \tilde QX) .
\end{align*}
We also have
\begin{eqnarray*}
  g(N^{(1)}(Y,Z),{f} X) = g({f}^2 [Y,Z] + [{f} Y, {f} Z] - {f}[{f} Y,Z] - {f}[Y,{f} Z], {f} X)\\
  = g({f}[Y,Z], \tilde Q X) + g([{f} Y, {f} Z] - {f}[{f} Y,Z] - {f}[Y,{f} Z] - [Y,Z], {f} X).
\end{eqnarray*}
From this and \eqref{3.3} we get the required result.
\end{proof}

\begin{remark}\rm
For particular values of the tensor $N^{(5)}$ we get
\begin{align}\label{KK}
\nonumber
 N^{(5)}(X,\xi_i,Z) & = -N^{(5)}(X, Z, \xi_i) = g( (\pounds_{\xi_i}{f})(Z)^\top,\, \tilde Q X),\\
\nonumber
 N^{(5)}(\xi_i,Y,Z) &= g([\xi_i, {f} Z]^\top, \tilde QY) -g([\xi_i,{f} Y]^\top, \tilde QZ),\\
 N^{(5)}(\xi_i,\xi_j,Y) &= N^{(5)}(\xi_i,Y,\xi_j)=0.
\end{align}
\end{remark}

We will discuss the meaning of $\nabla_{X}{f}$ for weak almost ${\cal S}$- and weak $K$- structures.
The following corollary of Proposition~\ref{lem6.1} and Theorem~\ref{thm6.2}
generalizes well-known results with $Q={\rm id}_{\,TM}$, e.g.,
\cite[Proposition 1.4]{b1970}, \cite[Proposition~3.3]{BP-2008A} and \cite[Proposition~2.1]{DIP-2001}.

\begin{corollary}\label{cor3.1}
For a weak almost ${\cal S}$-structure
we get
\begin{align}\label{3.1A}
\nonumber
 2\,g((\nabla_{X}{f})Y,Z) & = g(N^{(1)}(Y,Z),{f} X) +2\,g(fX,fY)\,\bar\eta(Z) -2\,g(fX,fZ)\,\bar\eta(Y) \\
 & + N^{(5)}(X,Y,Z),
\end{align}
where $\bar\eta=\sum\nolimits_{\,i}\eta^i$.
In particular, taking $x=\xi_i$ and then $Y=\xi_j$, we get
\begin{align}\label{3.1AA}
 2\,g((\nabla_{\xi_i}{f})Y,Z) &= N^{(5)}(\xi_i,Y,Z) ,\quad 1\le i\le p,\\
\label{2-xi}
 \nabla_{\xi_i}\,\xi_j & = 0,\qquad 1\le i,j\le p.
\end{align}
By \eqref{2-xi}, the characteristic distribution on a weak almost ${\cal S}$-manifold
is tangent to a totally geodesic foliation with zero sectional curvature: $K(\xi_i,\xi_j)=0$.
\end{corollary}

\begin{proof}
According to Theorem~\ref{thm6.2}, for a weak almost ${\cal S}$-structure we have
\[
 d\eta^i = \Phi,\quad N^{(2)}_i= N^{(4)}_{ij}=0.
\]
Thus, invoking \eqref{2.3} and using Theorem~\ref{thm6.2} in \eqref{3.1}, we get \eqref{3.1A}.

From \eqref{3.1AA} with $Y=\xi_j$ we get $g(f\nabla_{\xi_i}\,\xi_j, Z)=0$,
thus $\nabla_{\xi_i}\,\xi_j\in\ker f$.~Also,
\[
 \eta^k([\xi_i,\xi_j])= -2\,d\eta^k(\xi_i,\xi_j) =-2\,g(\xi_i, f\xi_j)=0;
\]
hence, $[\xi_i,\xi_j]=0$, i.e., $\nabla_{\xi_i}\,\xi_j=\nabla_{\xi_j}\,\xi_i$.
Finally, from $g(\xi_j,\xi_k)=\delta_{jk}$, using the covariant derivative with respect to $\xi_i$ and the above equality,
we get $\nabla_{\xi_i}\,\xi_j\in f(TM)$. This together with $\nabla_{\xi_i}\,\xi_j\in\ker f$ completes the proof of \eqref{2-xi}.
\end{proof}

The next corollary of Propositions~\ref{thm6.1} and \ref{lem6.1} generalizes \cite[Proposition 1.3]{b1970} with $Q={\rm id}_{\,TM}$.

\begin{corollary}
For a weak $K$-structure the covariant derivative of ${f}$ is given by
\begin{align}\label{3.1K}
\nonumber
 2\,g((\nabla_{X}{f})Y,Z) & =
 \sum\nolimits_{\,i}\big(  2\,d\eta^i({f} Y,X)\,\eta^i(Z) - 2\,d\eta^i({f} Z,X)\,\eta^i(Y) \\
 & +\eta^i([\tilde QY^\top,\,{f} Z])\,\eta^i(X) \big)
 + N^{(5)}(X,Y,Z).
\end{align}
In particular,
 $2\,g((\nabla_{\xi_i}{f})Y,Z) = \eta^i([\tilde QY^\top,\,{f} Z]) + N^{(5)}(\xi_i,Y,Z)$ for $1\le i\le p$.
\end{corollary}

\begin{proof}
 This follows directly from \eqref{3.1} and \eqref{3.1KK}.
\end{proof}

\begin{remark}\rm
We can rewrite \eqref{3.1}, \eqref{3.1A} and \eqref{3.1K} in terms of~$\Phi$
using the equality
\begin{equation}\label{E-nabla-Phi}
 (\nabla_{X}\Phi)(Z,Y)=g((\nabla_{X}{f})Y,Z)\,.
\end{equation}
\end{remark}

\section{The tensor field $h$}
\label{sec:3a}

Here, we apply for a weak almost ${\cal S}$-manifold the tensor field $h=(h_1,\ldots,h_p)$, where
\begin{align*}
 h_i=\frac{1}{2}\, N^{(3)}_i = \frac{1}{2}\,\pounds_{\xi_i}{f} .
\end{align*}
By Theorem~\ref{thm6.2}, $h_i=0$ if and only if $\xi_i$ is a Killing field.
First, we calculate
\begin{align}\label{4.2}
\nonumber
 (\pounds_{\xi_i}{f})X &\overset{\eqref{3.3B}} = \nabla_{\xi_i}({f} X) - \nabla_{{f} X}\,\xi_i - {f}(\nabla_{\xi_i}X - \nabla_{X}\,\xi_i)\notag\\
 &= (\nabla_{\xi_i}{f})X - \nabla_{{f} X}\,\xi_i + {f}\nabla_X\,\xi_i.
\end{align}
Taking $X=\xi_i$ in \eqref{4.2} and using $\nabla_{\xi_i}\,\xi_j=0$ (see Corollary~\ref{cor3.1}) and $(\nabla_{\xi_i}{f})\,\xi_j=\frac12 N^{(5)}(\xi_i,\xi_j,\,\cdot)=0$, see \eqref{3.1AA}, we get
\begin{align}\label{4.2b}
 h_i\,\xi_j = 0.
\end{align}

The next lemma proves stability of the results of \cite[Proposition~3.4]{BP-2008A} and \cite[Proposition~2.2]{DIP-2001}
that for an almost ${\cal S}$-structure, each tensor $h_i$ is self-adjoint and anticommutes with ${f}$.

\begin{proposition}\label{L3.1}
For a weak almost ${\cal S}$-structure, the tensor $h_i$ and its conjugate $h_i^*$ satisfy
\begin{eqnarray}
\label{E-31}
 (h_i-h_i^*)X &=& \frac{1}{2}\,N^{(5)}(\xi_i, X, \,\cdot),\\
 \label{E-30b}
 g(Q\,\nabla_{X}\,\xi_i, Z) &=& g(({f}+h_i{f}) Z,QX) - \frac 12\,N^{(5)}(X,\xi_i, {f}Z), \\
 \label{E-31A}
 h_i{f}+{f}\, h_i &=& -\frac12\,\pounds_{\xi_i}{\tilde Q} .
\end{eqnarray}
\end{proposition}

\begin{proof}
(i) The scalar product of \eqref{4.2} with $Y$ for $X,Y\in f(TM)$, using \eqref{3.1AA}, gives
\begin{align}\label{4.3}
 g((\pounds_{\xi_i}{f})X,Y) &= N^{(5)}(\xi_i, X, Y) + g({f}\nabla_{X}\,\xi_i - \nabla_{{f} X}\,\xi_i,\ Y).
\end{align}
Similarly,
\begin{align}\label{4.3b}
 g((\pounds_{\xi_i}{f})Y,X) &=
 N^{(5)}(\xi_i, Y, X) +g({f}\nabla_{Y}\,\xi_i - \nabla_{{f} Y}\,\xi_i,\ X).
\end{align}
Using \eqref{2.7X} and $(fX)(\eta^i(Y)) -(fY)(\eta^i(X))\equiv0$
(this vanishes if either $X$ or $Y$ equals $\xi_j$ and also for $X$ and $Y$ in~$f(TM)$), we get
\[
 N^{(2)}_i (X,Y) = \eta^i([f Y, X]-[f X, Y]) .
\]
Thus, the difference of \eqref{4.3} and \eqref{4.3b} gives
\begin{align*}
 2\,g((h_i-h_i^*)X,Y) =  N^{(5)}(\xi_i, X, Y) - N^{(2)}_i (X,Y).
\end{align*}
From this and equality $N^{(2)}_i=0$ (see Theorem~\ref{thm6.2}) we get \eqref{E-31}.

(ii) From Corollary \ref{cor3.1} with $Y=\xi_i$, we find
\begin{align}\label{4.4}
 2\,g((\nabla_{X}{f})\xi_i,Z) &= g(N^{(1)}(\xi_i,Z),{f} X)
 - 2\,g({f} Z , {f} X)
 + N^{(5)}(X,\xi_i,Z).
\end{align}
From \eqref{2.5} with $Y=\xi_i$, we get
\begin{align}\label{2.5B}
 [{f},{f}](X,\xi_i) = {f}^2 [X,\xi_i] - {f}[{f} X,\xi_i] .
\end{align}
Using \eqref{2.2}, \eqref{2.5B} and \eqref{3.3B}, we calculate
\begin{align}
\label{4.4A}
 g([{f},{f}](\xi_i,Z),{f} X)&= g({f}^2\,[\xi_i,Z] - {f}[\xi_i,{f} Z],{f} X) = - g({f}(\pounds_{\xi_i}{f})Z,{f} X)\notag\\
 &= - g((\pounds_{\xi_i}{f})Z,QX) +\sum\nolimits_{\,j}\eta^j(X)\,\eta^j((\pounds_{\xi_i}{f})Z) .
 \end{align}
Using \eqref{4.2}, we get
\begin{align}\label{3.1A3}
 2\,g((\nabla_{\xi_i}{f})Y,\xi_j) \overset{\eqref{3.1AA}}= N^{(5)}(\xi_i,Y,\xi_j) \overset{\eqref{KK}}=0,
\end{align}
From \eqref{4.2} and \eqref{3.1A3} we get
\begin{align}\label{4.5}
 g((\pounds_{\xi_i}{f})X,\xi_j) = -g(\nabla_{{f} X}\,\xi_i,\xi_j) .
\end{align}
Since ${f}\,\xi_i=0$, we find
\begin{align}\label{4.5A}
 (\nabla_{X}{f})\,\xi_i = -{f}\,\nabla_{X}\,\xi_i.
\end{align}
Thus, combining \eqref{4.4}, \eqref{4.4A} and \eqref{4.5}, we deduce
\begin{align}\label{4.6}
 -g({f}\,\nabla_{X}\,\xi_i, Z) & = -g(h_iZ,QX)
 -g(X,QZ) + \frac 12\,N^{(5)}(X,\xi_i,Z) \notag\\
 & + \sum\nolimits_{\,j}\eta^j(X)\,\eta^j(Z) -\frac12\sum\nolimits_{\,j}\eta^j(X)\,g(\nabla_{fZ}\,\xi_i, \xi_j).
\end{align}
From \eqref{2.3} we have
$g([X,\xi_i], \xi_k) = 2\,d\eta^k(\xi_i,X)=2\,\Phi(\xi_i,X)=0$.
By \eqref{2-xi}, we get
$g(\nabla_X\,\xi_i, \xi_k) = g(\nabla_{\xi_i}X, \xi_k) = -g(\nabla_{\xi_i}\xi_k, X) = 0$ for $X\in f(TM)$,
thus
\begin{align}\label{E-30-xi}
  g(\nabla_{X}\,\xi_i,\ \xi_k) = 0,\quad X\in TM,\ 1\le i,k \le p .
\end{align}
Replacing $Z$ by ${f} Z$ in \eqref{4.6} and using \eqref{2.1}
and ${f}\,\xi_i=0$, we achieve
\begin{eqnarray}\label{E-30}
\nonumber
 && g(Q\,\nabla_{X}\,\xi_i, Z) = g(({f}+h_i{f}) Z,QX) - \frac 12\,N^{(5)}(X,\xi_i, {f}Z) \\
 &+&\frac12\sum\nolimits_{\,j}\eta^j(X)\,g(\nabla_{f^2Z}\,\xi_i, \xi_j)
 -\sum\nolimits_{\,j}\eta^j(Z)\,g(\nabla_{X}\,\xi_i, \xi_j).
\end{eqnarray}
Using \eqref{E-30-xi} in \eqref{E-30} yields~\eqref{E-30b}.

(iii) Using \eqref{2.1}, we obtain
\begin{eqnarray*}
 & {f}\nabla_{\xi_i}{f} +(\nabla_{\xi_i}{f}){f} = \nabla_{\xi_i}\,({f}^2)
 = -\nabla_{\xi_i}\tilde Q +\nabla_{\xi_i}(\sum\nolimits_{\,j}\eta^j\otimes \xi_j) ,
\end{eqnarray*}
where in view of \eqref{2-xi}, we get $\nabla_{\xi_i}(\sum\nolimits_{\,j}\eta^j\otimes \xi_j)=0$.
By the above and \eqref{4.2}, we get \eqref{E-31A}:
\begin{eqnarray*}
 && 2(h_i{f}+{f} h_i)X  = {f}(\pounds_{\xi_i}{f})X +(\pounds_{\xi_i}{f}){f} X  \\
 && = {f}(\nabla_{\xi_i}{f})X +(\nabla_{\xi_i}{f}){f} X +{f}^2\nabla_X\,\xi_i -\nabla_{{f}^2 X}\,\xi_i \\
 && = -(\nabla_{\xi_i}\tilde Q)X -\tilde Q\nabla_X\,\xi_i+\nabla_{\tilde QX}\,\xi_i
 +\sum\nolimits_{\,j}\big(g(\nabla_X\,\xi_i, \xi_j)\,\xi_j -g(X, \xi_j)\nabla_{\xi_j}\,\xi_i\big) \\
 && = [\tilde QX, \xi_i] - \tilde Q\,[X, \xi_i] = -(\pounds_{\xi_i}{\tilde Q})X.
\end{eqnarray*}
Here, we used \eqref{2-xi} and \eqref{E-30-xi} to show $\sum\nolimits_{\,j}\big(g(\nabla_X\,\xi_i, \xi_j)\,\xi_j -g(X, \xi_j)\nabla_{\xi_j}\,\xi_i\big)=0$.
\end{proof}

\begin{remark}\rm
For a weak almost ${\cal S}$-structure, using \eqref{3.1A3}, we find
\[
 2\,g(h_i X, \xi_j)=-g(\nabla_{f X}\,\xi_i, \xi_j) \overset{\eqref{E-30-xi}}= 0;
\]
hence,
$f(TM)$ is invariant under $h_i$; moreover, $h^*_i\,\xi_j = 0$, see also \eqref{4.2b}.
\end{remark}

\begin{proposition}
For a weak ${\cal S}$-structure $({f},Q,\xi_i,\eta^i,g)$ we obtain
\begin{align}\label{E-5.1}
 N^{(5)}( Y,\xi_i, {f} X) - N^{(5)}( X,\xi_i, {f} Y) = 2\,g( (h_i^*{f} + {f} h_i){f} X,\ {f} Y) .
\end{align}
\end{proposition}

\begin{proof}
Using the equality
\begin{align}\label{E-d-eta}
 2\,d\eta^i(X,Y) = g(\nabla_X\,\xi_i,\,Y) -g(\nabla_Y\,\xi_i,\,X),
\end{align}
we find
\begin{align*}
 2\,d\eta^i(X,Y) - 2\,g(X, {f} Y) & = g(\nabla_X\,\xi_i, Y)-g(\nabla_Y\,\xi_i, X) - 2\,g(X, {f} Y) \\
                                    & = g(Q\nabla_X\,\xi_i, \tilde Y)-g(Q\nabla_Y\,\xi_i, \tilde X) - 2\,g(X, {f} Y),
\end{align*}
where $X=Q\tilde X$ and $Y=Q\tilde Y$ and $X,Y\in \mathfrak{X}_M$.
From this, using \eqref{E-30} and the equality
\[
 g({f}\tilde Y, QX)-g({f}\tilde X, QY)=2\,g(X, {f} Y),
\]
we get
\begin{align}\label{E-5.01}
 0 = 2\,d\eta^i(X,Y) - 2\,g(X, {f} Y) & = g((h_i{f})\tilde Y, Q\tilde X)-g((h_i{f})\tilde X, Q\tilde Y) \notag\\
 & +\frac12\big\{N^{(5)}(\tilde X,\xi_i, {f}\tilde Y) - N^{(5)}(\tilde Y,\xi_i, {f}\tilde X)\big\} .
\end{align}
Next, we find
\begin{align}\label{E-5.00}
 Qh_i= -{f}^2 h_i = -{f}(A_i-h_i{f}) = -{f} A_i +(A_i-h_i{f}){f} = h_iQ+[A_i,{f}],
\end{align}
where $A_i=h_i{f}+{f} h_i$ and $B_i= h_i^* - h_i$, see Proposition~\ref{L3.1},
From \eqref{E-5.01}, using \eqref{E-5.00}, we get \eqref{E-5.1}.
\end{proof}



\section{The rigidity of an ${\cal S}$-structure}
\label{sec:3}

An important class of metric $f$-manifolds is given by ${\cal S}$-manifolds.
Here, we study a wider class of weak ${\cal S}$-manifolds and prove their rigidity in the class of metric weak $f$-manifolds.

\hskip-4pt
The following result generalizes \cite[Proposition~1.4]{b1970} and \cite[Proposition~3.11]{BP-2008A}.

\begin{proposition}
For a weak ${\cal S}$-structure we get
\begin{align}\label{4.10}
 g((\nabla_{X}{f})Y,Z) & = g(QX,Y)\,\bar\eta(Z) - g(QX,Z)\,\bar\eta(Y) +\frac{1}{2}\, N^{(5)}(X,Y,Z) \notag\\
 & +\sum\nolimits_{\,j} \eta^j(X)\big(\bar\eta(Y)\eta^j(Z) - \eta^j(Y)\bar\eta(Z)\big) ,
\end{align}
or, using dual to $\bar\eta$ vector $\bar\xi=\sum\nolimits_{\,i}\xi_i$,
\begin{equation}\label{4.11}
 (\nabla_{X}{f})Y = g(fX,fY)\,\bar\xi + \bar\eta(Y)f^2 X + \frac{1}{2}\,N^{(5)}(X,Y,\,\cdot).
\end{equation}
\end{proposition}

\begin{proof} Since $({f},Q,\xi_i,\eta^i,g)$ is a metric weak $f$-structure with $N^{(1)}=0$, by Corollary~\ref{cor3.1}, we get~\eqref{4.11}.
Using \eqref{2.2} in the scalar product of \eqref{4.11} with $Z$, we obtain \eqref{4.10}.
\end{proof}

Using the equality \eqref{E-nabla-Phi}, in \eqref{4.10}, we can write the condition of a weak ${\cal S}$-structure in terms of~$\Phi$.

A weak almost ${\cal S}$-structure with positive partial Ricci curvature can be deformed to an almost ${\cal S}$-structure, see~\cite{RWo-2} .
The main result in this section is the following rigidity of an ${\cal S}$-structure.

\begin{theorem}\label{T-4.1}
A metric weak $f$-structure
is a weak ${\cal S}$-structure if and only if it is an ${\cal S}$-structure.
\end{theorem}

\begin{proof}
Let $({f},Q,\xi_i,\eta^i,g)$ be a weak ${\cal S}$-structure. Replacing $Y$ by $\xi_i$ in \eqref{4.10}, we get
\begin{align}\label{4.12}
\nonumber
 g((\nabla_{X}{f})\,\xi_i,Z) &= \eta^i(X)\,\bar\eta(Z) - g(QX,Z) +\frac{1}{2}\,N^{(5)}(X,\xi_i,Z) \\
 & = - g(Q X^\top, Z) +\frac{1}{2}\,N^{(5)}(X,\xi_i,Z).
\end{align}
Using \eqref{4.5A}, we rewrite \eqref{4.12} as
\begin{align*}
 g(\nabla_{X}\,\xi_i,{f} Z) = - g(Q X^\top,Z) +\frac{1}{2}\, N^{(5)}(X,\xi_i,Z).
\end{align*}
By the above and \eqref{2.1}, we find
\begin{align}\label{4.14}
 g(\nabla_X\,\xi_i +{f} X^\top, \,{f}\,Z) = \frac12\,N^{(5)}(X,\xi_i,Z).
\end{align}
Take $\tilde X,\tilde Y\in f(TM)$ such that ${f}\tilde X = X^\top$ and ${f}\,\tilde Y=Y^\top$.
Using \eqref{E-d-eta} and keeping in mind \eqref{4.14} and \eqref{2.3}, we get
\begin{align}\label{4.15}
 d\eta^i(X,Y) - g(X, {f} Y) = \frac14\,\big\{N^{(5)}({f}\tilde X,\xi_i, \tilde Y) - N^{(5)}({f}\tilde Y,\xi_i, \tilde X)\big\}.
\end{align}
By \eqref{4.15}, we get that \eqref{2.3} holds if the bilinear form $(X,Y) \to N^{(5)}({f} X,\xi_i, Y)$ is symmetric:
\begin{align}\label{4.16}
 & N^{(5)}({f} X,\xi_i, Y) = N^{(5)}({f} Y,\xi_i, X) .
\end{align}
Since ${f}$ is skew-symmetric, applying \eqref{4.10} with $Z=\xi_i$ in \eqref{4.NN}, we obtain
\begin{align}\label{4.17}
& g( [{f},{f}](X,Y),\xi_i) = - g((\nabla_{{f} Y}{f})X, \xi_i)  +g((\nabla_{{f} X}{f})Y, \xi_i) \notag\\
& = -g(Q\,{f} Y,X) + g(Q\,{f} Y, \xi_i)\,\eta^i(X) - \frac{1}{2}\, N^{(5)}({f} Y, X, \xi_i)\notag\\
&  +g(Q\,{f} X,Y) -g(Q\,{f} X, \xi_i)\,\eta^i(Y)+\frac{1}{2}\,N^{(5)}({f} X, Y, \xi_i) .
\end{align}
Recall that $[Q,\,{f}]=0$ and $f\,\xi_i=0$. Thus, \eqref{4.17} yields
\begin{align*}
  g( [{f},{f}](X,Y), \xi_i) = - 2\,g(QX,{f} Y)
   +\frac{1}{2}\big\{ N^{(5)}({f} Y, \xi_i, X) -N^{(5)}({f} X, \xi_i, Y) \big\}.
\end{align*}
From this, using definition of $N^{(1)}$, we get
\begin{align}\label{4.18}
 g(N^{(1)}(X,Y), \xi_i)  = -2\,g(\tilde Q X, {f} Y) + \frac{1}{2}\big\{ N^{(5)}({f} Y, \xi_i, X) - N^{(5)}({f} X, \xi_i, Y) \big\}.
\end{align}
By \eqref{4.18}, if $N^{(1)}(X,Y)=0$ holds then the following condition is valid for all $X,Y\in \mathfrak{X}_M$:
\begin{align}\label{4.19}
  N^{(5)}({f} Y, \xi_i, X) -N^{(5)}({f} X, \xi_i, Y) =  4\,g(\tilde Q X, {f} Y) .
\end{align}
In view of \eqref{4.16}, equation \eqref{4.19} reduces to
\begin{align*}
 0 = g(\tilde Q X, {f} Y) = g(X - QX, {f} Y),
\end{align*}
thus, $Q\,|_{\,f(TM)}={\rm id}|_{\,f(TM)}$.
\end{proof}

For a weak almost ${\cal S}$-structure all $\xi_i$ are Killing if and only if $h=0$, see Theorem~\ref{thm6.2}.
The equality $h=0$ holds for a weak ${\cal S}$-structure since it is true for an ${\cal S}$-structure, see Theorem~\ref{T-4.1}.
We~will prove this property of the weak ${\cal S}$-structure directly.

\begin{corollary}
For a weak ${\cal S}$-structure, $\xi_1,\ldots,\xi_p$ are Killing vector fields; moreover, $\ker f$ is integrable and defines a Riemannian totally geodesic foliation.
\end{corollary}

\begin{proof}
In view of \eqref{4.5A} and $\bar\eta(\xi_i)=1$, Eq. \eqref{4.10} with $Y=\xi_i$ becomes
\begin{align}\label{4.6A}
 g(\nabla_{X}\,\xi_i, {f} Z) = \eta^i(X)\,\bar\eta(Z) -g(X,QZ) + \frac 12\,N^{(5)}(X,\xi_i,Z) .
\end{align}
Combining \eqref{4.6} and \eqref{4.6A}, and using \eqref{E-30-xi}, we achieve for all $i$ and $X,Z$,
\begin{align*}
 g(h_iZ,QX) & = \sum\nolimits_{\,j}\eta^j(X)\,\eta^j(Z) -\eta^i(X)\,\bar\eta(Z) ,
\end{align*}
which implies $hZ=0$ for $Z\in f(TM)$ (since $Q$ is nonsingular). From this and \eqref{4.2b} we get $h=0$.
By Theorem~\ref{thm6.2}, $\ker f$ defines a totally geodesic foliation. Since $\xi_i$ is a Killing field,
we get
\[
 0 = (\pounds_{\xi_i}\,g)(X,Y) = g(\nabla_{X}\,\xi_i, Y) + g(\nabla_{Y}\,\xi_i, X)
 = -g(\nabla_{X} Y + \nabla_{Y} X,\ \xi_i)
\]
for all $i$ and $X,Y\bot\,\ker f$. Thus, $f(TM)$ is totally geodesic, i.e., $\ker f$ defines a Riemannian foliation.
\end{proof}

For $p=1$, from Theorem~\ref{T-4.1} we have the following

\begin{corollary}[see \cite{RovP-arxiv}]
A weak almost contact metric structure on $M^{2n+1}$ is weak Sasakian if and only if it is a Sasakian structure
(i.e., a normal weak contact metric structure) on $M^{2n+1}$.
\end{corollary}

\section{The characteristic of a weak ${\cal C}$-structure}
\label{sec:4}

An important class of metric $f$-manifolds is given by ${\cal C}$-manifolds.
Here, we study a wider class of weak almost ${\cal C}$-manifolds and
characterize weak ${\cal C}$-manifolds in the class of metric weak $f$-manifolds using the condition $\nabla{f}=0$.

\begin{proposition}
Let $({f},Q,\xi_i,\eta^i,g)$ be a weak
${\cal C}$-structure. Then
\begin{align}\label{6.1}
 2\,g((\nabla_{X}{f})Y,Z) &= N^{(5)}(X,Y,Z),\\
\label{6.1b}
 0 & = N^{(5)}(X,Y,Z) + N^{(5)}(Y,Z,X) + N^{(5)}(Z, X, Y) ,\\
\label{6.1c}
  0 & = N^{(5)}({f} X,Y,Z) + N^{(5)}({f} Y,Z,X) + N^{(5)}({f} Z, X, Y) .
\end{align}
In particular, using \eqref{6.1} with $Y=\xi_i$ and \eqref{2.1}, we get
\begin{align*}
 g(\nabla_{X}\,\xi_i,\,Q Z) = -\frac12\,N^{(5)}(X,\xi_i,{f} Z) .
\end{align*}
\end{proposition}

\begin{proof}
For a weak almost ${\cal C}$-structure $({f},Q,\xi_i,\eta^i,g)$, using Theorem~\ref{thm6.2C},  from \eqref{3.1} we get
\begin{equation}\label{6.1a}
 2\,g((\nabla_{X}{f})Y,Z)= g([{f},{f}](Y,Z),{f} X) + N^{(5)}(X,Y,Z).
\end{equation}
From \eqref{6.1a}, using condition $[{f},{f}]=0$ we get \eqref{6.1}.
Using \eqref{3.3} and \eqref{6.1}, we can write
\[
  0 = 3\,d\Phi(X,Y,Z) = g((\nabla_X\,{f})Z,Y) +g((\nabla_Y\,{f})X, Z) +g((\nabla_Z\,{f})Y, X),
\]
hence, \eqref{6.1b}.
Using \eqref{4.NN}, \eqref{6.1} and the skew-symmetry of ${f}$, we obtain
\begin{align*}
 0 = 2\,g([{f},{f}](X,Y),Z) & = N^{(5)}(X, Y, {f} Z) + N^{(5)}({f} X, Y, Z) \\
 & - N^{(5)}(Y, X, {f} Z) - N^{(5)}({f} Y,X,Z) .
\end{align*}
This and \eqref{6.1b} with $X$ replaced by ${f} X$ provide \eqref{6.1c}.
\end{proof}

Consider a weaker condition than \eqref{6.1d}:
\begin{align}\label{E-xi31}
 [\xi_i,\xi_j]^\bot =0,\quad 1\le i,j\le p.
\end{align}
Recall that a ${K}$-structure is a ${\cal C}$-structure if and only if ${f}$ is parallel, e.g., \cite[Theorem~1.5]{b1970}
(and \cite[Theorem~6.8]{blair2010riemannian} for $p=1)$.
 The following our theorem extends these results.

\begin{theorem}\label{thm6.2D}
A metric weak $f$-structure with $\nabla{f}{=}\,0$ and \eqref{E-xi31} is a~weak ${\cal C}$-structure with $N^{(5)}=0$.
\end{theorem}

\begin{proof}
Using condition $\nabla{f}=0$, from \eqref{4.NN} we obtain $[{f},{f}]=0$.
Hence, from \eqref{2.6X} we get $N^{(1)}(X,Y)=2\,\sum\nolimits_{\,i} d\eta^i(X,Y)\,\xi_i$,
and from \eqref{4.NNxi} we obtain
\begin{align}\label{E-cond1}
 \nabla_{{f} X}\,\xi_i - {f}\,\nabla_{X}\,\xi_i = 0,\quad X\in \mathfrak{X}_M.
\end{align}
From \eqref{3.3}, we calculate
\[
 3\,d\Phi(X,Y,Z) = g((\nabla_{X}{f})Z, Y) + g((\nabla_{Y}{f})X,Z) + g((\nabla_{Z}{f})Y,X);
\]
hence, using condition $\nabla{f}=0$ again, we get $d\Phi=0$. Next,
\begin{align*}
 N^{(2)}_i(Y,\xi_j) = -\eta^i([{f} Y,\xi_j]) = g(\xi_j, {f}\nabla_{\xi_i} Y) =0.
\end{align*}
Thus, setting $Z=\xi_i$ in Proposition~\ref{lem6.1} and using the condition $\nabla{f}=0$ and the properties
$d\Phi=0$, $N^{(2)}_i(Y,\xi_j)=0$ and $N^{(1)}(X,Y)=2\sum\nolimits_{\,i} d\eta^i(X,Y)\,\xi_i$, we find
 $0 = d\eta^i({f} Y, X) - N^{(5)}(X,\xi_i, Y)$.
 By~\eqref{KK} and \eqref{E-cond1}, we get
\[
 N^{(5)}(X,\xi_i, Y) = g([\xi_i,{f} Y]^\top -{f}[\xi_i,Y],\, \tilde Q X)
 = g(\nabla_{{f} Y}\,\xi_i - {f}\,\nabla_{Y}\,\xi_i,\, \tilde Q X) = 0;
\]
hence, $d\eta^i({f} Y, X)=0$. From this and condition
\[
 g([\xi_i,\xi_j],\xi_k)=2\,d\eta^k(\xi_j, \xi_i)=0
\]
we get $d\eta^i=0$. By the above, $N^{(1)}=0$.
Thus, $({f},Q,\xi_i,\eta^i,g)$ is a weak ${\cal C}$-structure.
Finally, from \eqref{6.1} and condition $\nabla{f}=0$ we get $N^{(5)}=0$.
\end{proof}

\begin{example}\rm
Let $M$ be a $2n$-dimensional smooth manifold and $\tilde{f}:TM\to TM$ an endomorphism of rank $2n$ such that
$\nabla\tilde{f}=0$.
To construct a weak ${\cal C}$-structure on $M\times\  mathbb{R}^p$ or $M\times \mathbb{T}^p$, where
$\mathbb{R}^p$ is a Euclidean space and $\mathbb{T}^p$ is a $p$-dimensional flat torus,
take any point $(x, t_1,\ldots,t_p)$ of either
space and set $\xi_i = (0, d/dt_i)$, $\eta^i =(0, dt_i)$~and
\[
 {f}(X, Y) = (\tilde{f} X, 0),\quad
 Q(X, Y) = (-\tilde{f}^{\,2} X,\, Y).
\]
where $X\in T_xM$ and $Y= \sum_i Y^i\xi_i\in\{\mathbb{R}^p_t, \mathbb{T}^p_t\}$.
Then \eqref{2.1} holds and Theorem~\ref{thm6.2D} can be used.
\end{example}

For $p=1$, from Theorem~\ref{thm6.2D} we have the following

\begin{corollary}[see \cite{RovP-arxiv}]
Any weak almost contact structure $(\varphi,Q,\xi,\eta,g)$ with the property $\nabla\varphi=0$
is a~weak cosymplectic structure, i.e., $d\Phi=0$ and $d\eta=0$, with vanishing tensor $N^{(5)}$.
\end{corollary}


\begin{thebibliography}{00}


\bibitem{AM-1995}
D. Alekseevsky and P. Michor, {Differential geometry of $\mathfrak{g}$-manifolds}, Differential Geom. Appl. {5} (1995), 371--403

\bibitem{blair2010riemannian}
D. Blair, {Riemannian geometry of contact and symplectic manifolds}, Springer, 2010

\bibitem{b1970}
D. Blair, Geometry of manifolds with structural group $U(n)\times O(s)$, J. Diff. Geom. 4 (1970), 155--167

\bibitem{BP-2008A}
L. Brunetti and A.\,M. Pastore, Curvature of a class of indefinite globally framed $f$-manifolds.
Bull. Math. Soc. Sci. Math. Roumanie, Tome 51(99) No. 3 (2008), 183--204

\bibitem{BN-1985}
A. Bucki and A. Miernowski, Almost $r$-paracontact structures. Ann. Univ. Mariae Curie-Sklodowska 39\,(2) (1985), 13--26

\bibitem{ca-toh}
G. Cairns, {A general description of totally geodesic foliations},
Tohoku Math. J. {38} (1986), 37--55

\bibitem{DIP-2001}
K.\,L. Duggal, S. Ianus, and A.\,M. Pastore, Maps interchanging $f$-structures and their harmonicity. Acta Appl. Math. 67 (2001), 91--115

\bibitem{fip}
M. Falcitelli, S. Ianus, and A.M. Pastore, Riemannian Submersions and Related Topics, World Scientific, 2004

\bibitem{FP-2017}
L.\,M. Fern\'{a}ndez and A. Prieto-Mart\'{i}n, On $\eta$-Einstein para-$S$-manifolds.
Bull. Malays. Math. Sci. Soc. 40  (2017), 1623--1637


\bibitem{gy}
S.\,I. Goldberg and K. Yano, On normal globally framed $f$-manifolds, Tohoku Math. J. 22 (1970), 362--370

\bibitem{RovP-arxiv}
V. Rovenski, and D.\,S. Patra, On the rigidity of the Sasakian structure and characterization of cosymplectic manifolds,
arXiv:2203.04597\,v2, 2022, 15\,pp.

\bibitem{RWo-2}
V. Rovenski, and R. Wolak, {New metric structures on $\mathfrak{g}$-foliations},
Indagationes Mathema\-ticae, 33 (2022), 518--532

\bibitem{Rov-Wa-2021}
V. Rovenski, and P.\,G. Walczak, {Extrinsic geometry of foliations}. Progress in Mathematics, vol. 339, Birkh\"{a}user, 2021

\bibitem{tpw}
L. Di Terlizzi, A.M. Pastore, R.Wolak,
Harmonic and holomorphic vector fields on an $f$-manifold with parallelizable kernel.
An. Stiint. Univ. Al. I. Cuza Iausi, Ser. Noua, Mat. 60, No. 1 (2014), 125--144

\bibitem{yan}
K. Yano, On a structure $f$ satisfying $f^3+f=0$, Technical Report No. 12, University of Washington, 1961

\end{thebibliography}
\end{document}